\documentclass{tsp-short}
\theoremstyle{plain}
\newtheorem{thm}{Theorem}[section]
\newtheorem{lem}{Lemma}[section]

\theoremstyle{definition}
\newtheorem{defn}{Definition}[section]

\theoremstyle{remark}
\newtheorem{remk}{Remark}[section]

\newcommand{\1}{1\!\!\!\!\!\;{\rm I}}
\newcommand{\mbR}{{\mathbb R}}
\newcommand{\mbN}{{\mathbb N}}

\newcommand{\mbZ}{{\mathbb Z}}
\newcommand{\lin}{\underset{n\to\infty}{\lim}}

\newcommand{\cF}{{\mathcal F}}

\newcommand{\Pb}{\mathrm{P}}
\newcommand{\E}{\mathrm{E}}

\newcommand{\ov}{\overline}

\newcommand{\vf}{\varphi}

\newcommand{\ve}{\varepsilon}

\newcommand{\be}{\begin{equation}}
\newcommand{\ee}{\end{equation}}

\title{On a limit behavior of a   random walk with modifications at zero}
\author{Andrey Pilipenko}
\address{Institute of Mathematics, National
Academy of Sciences of Ukraine, Kyiv, Ukraine; National Technical University of Ukraine ``Kyiv Polytechnical Institute''}
 \email{ pilipenko.ay@yandex.ua}
\author{Vladislav Khomenko}
\address{National Technical University of Ukraine ``Kyiv Polytechnical Institute''}
 \email{khomenko.vlad7@gmail.com}


\begin{document}

\begin{abstract}
We consider the limit behavior of a one-dimensional random walk with  unit jumps whose transition probabilities are modified every time 
the walk hits zero. 
The  invariance principle is proved in the scheme of series where the size of modifications depends on the number of series.
For the natural scaling of time and  space  arguments  the limit process is (i) a   Brownian motion if modifications are ``small'', (ii)
a linear motion with a random slope if modifications are ``large'', and (iii) the limit process satisfies an SDE with a local time of unknown process 
in a drift if modifications are ``moderate''.
\end{abstract}

\maketitle

\section{Introduction and Main Results}
 Consider a random walk $\{X_n, \ n\geq 0\}$ on $\mbZ$ with  unit jumps that is constructed in the following way. 
It behaves as a symmetric random walk until
 the first visit to 0. After that the probability of the jump to the right becomes equal to   $p_1:=1/2+\Delta$, and to the left $q_1:=1/2-\Delta,$ where $\Delta>0$ is a fixed number.
 When $\{X_n\}$ secondly
 visits 0 its transition probabilities to the right and to the left become equal to  $p_2:=1/2+2\Delta$ and  $q_2:=1/2-2\Delta,$ respectively, etc. (if $1/2+k\Delta>1$ we set $p_k:=1$).
 
 Let us give the formal definition.
 \begin{defn}
 A random sequence $\{X_n, \ n\geq 0\}$ with values in $\mbZ$ is called a random walk with  modifications (RWM) at 0
 if
$$
\forall k\geq 1\ \forall i_0, i_1,\dots,i_k, \ |i_{j+1}-i_j|=1
$$ 
$$
\Pb(X_{k+1}=i_k+1\ |\ X_0=i_0, X_1=i_1,\dots, X_k=i_k) = (\frac12+\nu_k\Delta)\wedge 1,
$$
$$
\Pb(X_{k+1}=i_k-1\ |\ X_0=i_0, X_1=i_1,\dots, X_k=i_k) = (\frac12-\nu_k\Delta)\vee 0,
$$
where $\nu_k=|\{ j\in \ov{0,k}\ : \ X_j=0\}|=\sum_{j=0}^k\1_{\{X_j=0\}}$ is the number of visits  to 0.

The number $\Delta>0$ is called the size of modifications. 
\end{defn}
 Set $\cF_n:=\sigma(X_0,X_1,\dots,X_n)$. The previous definition is equivalent to
 $$
 P(X_{k+1}=X_k+1\ |\ \cF_n)= (\frac12+\nu_k\Delta)\wedge 1,
 $$
 $$
 P(X_{k+1}=X_k-1\ |\ \cF_n)= (\frac12-\nu_k\Delta)\vee 0.
 $$
 \begin{remk}
 The usual random walk with  unit jumps and fixed transition probabilities $p$ and $(1-p)$ is non-recurrent if $p\neq 1/2.$ So, 
 $1/2+\nu_\infty \Delta<1$ with positive probability, where $\nu_\infty:=|\{ n\geq 0\ : \ X_n=0\}|$.
 \end{remk}
 The aim of the paper is to study the limit behavior of the sequence of series $\{X^{(n)}_k\}$ where the size of modifications in the $n$-th series
 $\Delta_n\to0$ as $ n\to\infty.$
 
 It is well-known that if $\Delta=0$, i.e., if $\{X_k\}$ is a symmetric random walk with the unit jumps, then the sequence of processes
 $\{\frac{X_{[n\cdot]}}{\sqrt{n}}\}$ converges in distribution to a Brownian motion   in the space $D([0,\infty)).$ So, it is natural to expect that if $\Delta_n\to0$ fast enough,
 then the limit of $\{\frac{X_{[n\cdot]}}{\sqrt{n}}\}$ will be a Brownian motion too. On the other hand, if $\{Y_k\}$ is a random walk with $p_{i,i+1}=p, p_{i,i-1}=1-p,$ then by the law of large numbers we have a.s. convergence
\be\label{eq178}
 \lin \frac{Y_{[nt]}}n= (1-2p)t
 \ee
 for fixed $t\geq 0$ (and even uniformly on compact sets). Hence, if $\Delta_n\to 0$ ``slowly'', there is a possibility that some scaling of 
$X^{(n)}_{[nt]}$ converges to non-zero linear process with a random slope.

The main result of the paper is the following theorem.
\begin{thm}
Let $\Delta_n=\frac{c}{n^\alpha},$ where $c>0, \alpha>0.$ Assume that $X^{(n)}_0=0$ for all $n$, and $X^{(n)}_k$ is extended to all $t\geq 0$
by linearity
$$
X^{(n)}_t:= X^{(n)}_{[t]}+ (X^{(n)}_{[t+1]}-X^{(n)}_{[t]})(t-[t]).
$$
\begin{itemize}
\item  If $\alpha>1,$ then 
\be\label{eq192}
\frac{X^{(n)}_{nt}}{\sqrt{n}}\Rightarrow W(t), \ n\to\infty,
\ee
where $W$ is a  Brownian motion.
\item  If $0<\alpha<1,$ then 
\be\label{eq197}
\frac{X^{(n)}_{nt}}{n^{1-\frac\alpha{2}}}\Rightarrow 2\sqrt{c}\eta t, \ n\to\infty,
\ee
where $\eta$ is a non-negative random variable with the distribution function
\be\label{eq201}
\Pb(\eta\leq x)=1-e^{-\frac{x^2}2},\ \ x\geq 0.
\ee
\item  If $\alpha=1,$ then 
$$
\frac{X^{(n)}_{nt}}{\sqrt{n}}\Rightarrow X_\infty(t), \ n\to\infty,
$$
where $X_\infty$ satisfies the SDE
\be\label{eq208}
X_\infty(t)=\sqrt{c}\int_0^t l^0_{X_\infty} (s)ds+W(t), t\geq 0,
\ee
$l^0_{X_\infty} (t)=  {\underset{\ve\to0+}{\lim}}\frac{1}{2\ve}\int_0^t\1_{|{X_\infty} (s)|\leq \ve} ds$ is the local time of ${X_\infty}$ at 0.

Here $\Rightarrow $ denotes the weak convergence in the space $C([0,\infty)).$
\end{itemize}
\end{thm}
 \begin{remk}
 The case $X^{(n)}_0=x_n$ can be treated with the natural modifications.
 \end{remk}
 \begin{remk}
 Equation \eqref{eq208} has a unique weak solution due to Girsanov's theorem.
 \end{remk}
  \begin{remk}
 The fact that the case $\alpha=1$ is the critical one can be guessed by the following non-rigorous observations. 
 In some sense the sequence $\{X^{(n)}_k\}$ visits to 0 more  rare than the symmetric random walk with the unit jump (it may not return at all).
 The number of visits to 0 by the symmetric random walk has a rate $\sqrt{n}.$ So, if $\alpha>1,$
then 
$$
\max_{k=\ov{0,n}}|(\frac12+ \nu_k\Delta_n)-\frac12|=|(\frac12+ \nu_n\Delta_n)-\frac12|=\nu_n\Delta_n=\frac{c\nu_n}{n^\alpha} = O(n^{\frac12-\alpha}),
\ n\to\infty.
$$
If all transition probabilities where constant, i.e.,
\be\label{eq232}
p^{X^{(n)}}_{i,i+1}=\frac12+Kn^{\frac12-\alpha}, \ p^{X^{(n)}}_{i,i-1}=\frac12-Kn^{\frac12-\alpha},\ i\in\mbZ,
\ee
then it is not difficult to show \eqref{eq192}.  In some sense transition probabilities of RWM differ from $1/2$ even less than above.

On the other hand, if $\alpha<1$  and if transition probabilities are given in \eqref{eq232}, then $\E X^{(n)}_n/\sqrt{n}\to\infty.$
So  $n^{\frac12}$ is not a natural normalizing factor.  We will show that the total number of returns to  0 has a rate $n^{\alpha/2} $ and the instant  of the 
So  $n^{\frac12}$ is not a natural normalizing factor.  We will show that the total number of returns to  0 has a rate $n^{\alpha/2} $ and the instant  of the 
 last return to  0 of the process $X^{(n)}_{nt}$ converges  to  0 as $n\to \infty$. Therefore the natural choice for scaling is 
$$
n (\mbox{ steps }) \times \nu_\infty (\mbox{ number of modifications }) \times n^{-\alpha} (\mbox{ size of each modification }) \asymp
$$
$$
 \asymp \frac{n\,\cdot n^{\alpha/2}}{n^\alpha}= n^{1-\alpha/2},
$$
  Since  transition probabilities do not change after small amount of time (after the last return to 0), 
 the limit process should be linear (compare with \eqref{eq178}).
\end{remk}
\begin{remk}
 RWM is not a Markov chain because transition 
 probabilities depend on number of visits to 0. The process $X_\infty$ from \eqref{eq208} is  not  a Markov process too.
 However the pairs $\{(X_k,\nu_k), k\geq 0\}$, $\{({X_\infty} (t),l^0_{X_\infty} (t)), t\geq 0\}$ have  Markov property. 
\end{remk} 
\begin{remk}
For any $a$  we have   $\lim_{t\to\infty}\frac{at+W(t)}{t}=a$  a.s. Since the local time is non-decreasing non-negative function, it can be easily verified that $\lim_{t\to\infty}X_\infty(t)=+\infty$ a.s., $\Pb(\exists  t_0\ \forall t\geq t_0\ :\ \ l^0_{X_\infty}(t)=l^0_{X_\infty}(t_0))=1$, and
 $\lim_{t\to\infty}\frac{X_\infty(t)}t=\lim_{t\to\infty} l^0_{X_\infty}(t)>0$ a.s.,   where $X_\infty$ satisfies  \eqref{eq208}. It can be seen from the proof that  the distribution of $\lim_{t\to\infty} l^0_{X_\infty}(t)$    coincides with the distribution of $2\sqrt{c}\eta $, where the distribution function of  $\eta$ is given in \eqref{eq201}.
\end{remk} 
\begin{remk}
If transition probabilities were perturbed only at 0 and would not be changed in time, say
$p_{i,i\pm 1}=1/2$ for $i\neq 0$, $p_{0,1}=p,\ p_{0,-1}=1-p$,  then the weak limit of $\{\frac{X_{nt}}{\sqrt{n}}\}$ may be the skew Brownian motion, i.e., a solution of the SDE
$$
dX^{skew}(t)= ({2p-1})dl_{X^{skew}}^0(t) +dW(t),
$$
see \cite{Harrison+Shepp:1981} for this particular case, and \cite{IP2015, Minlos+Zhizhina:1997,  PP2015, Pilipenko+Prihodko:2014, PSakhanenko, Szasz+Telcs:1981} for further generalizations.

RWM also resembles the  multi-excited random walk (but does not equal) that is defined in the following way:
$$
\Pb(X^{ex}_{k+1}=X^{ex}_k+1\ | \ X^{ex}_0, X^{ex}_1,\dots, X^{ex}_k)=
$$
$$
1-\Pb(X^{ex}_{k+1}=X^{ex}_k-1\ | \ X^{ex}_0, X^{ex}_1,\dots, X^{ex}_k)= \frac12+\ve_{nj},
$$
if $j=|\{0\leq i\leq k, \ X^{ex}_i=X^{ex}_k\}|$ and $\{\ve_{nj}\}$ are some (may be random) variables. 

Under various assumptions on $\{\ve_{nj}\}$ and scaling, limits of multi-excited random walk may be a linear process, or more intricate processes, for example the limits may be a solution of the following stochastic equation
$$
dX^{ex}_\infty(t)= \vf(l^{X^{ex}_\infty(t)}_{X^{ex}_\infty}(t)) dt +dW(t)
$$
or
$$
X^{ex}_\infty(t)= \alpha\max_{s\in [0,t]}X^{ex}_\infty(s)-\beta \min_{s\in [0,t]}X^{ex}_\infty(s)+ W(t),
$$
see \cite{ Dolgopyat, KosyginaZerner, Raimond, Zerner} and references therein.
\end{remk} 

 \section{Auxiliary lemmas}
 
 Let ${X_k=X_k^\Delta}$ be an RWM, where the modification equals  $\Delta>0.$ For simplicity assume that $X_0=0.$ 
 
 Set $\nu_k=\nu^\Delta_k=|\{i=\ov{0,k}
\ :\ X_i=0\}|,\ \nu_\infty=\nu^\Delta_\infty=|\{i\geq 0\ :\ X_i=0\}|$.
 
\begin{lem}\label{lem2.1}
\be\label{eq284}
\Pb\left(\nu_\infty^\Delta\geq k\right)=\underset{i=1}{\overset{k}\Pi} (1-2i\Delta), \ k\in\mbN.
\ee
\end{lem}
\begin{proof}
It is well known that if $\{S_k\}$ is a random walk with  unit jumps, $p_{i,i+1}=1-p_{i,i-1}=p,$
then $\Pb(\exists k\geq 1  \ S_k=0\ | \ S_0=0)=1-|p-q|=1-|2p-1|$.

So, if $\frac12+k\Delta\leq 1,$ then 
$$
\Pb(\nu_\infty\geq k+1\ | \ \nu_\infty\geq k)=1-\left(2(\frac12+k\Delta)-1\right)=1-2 k\Delta.
$$
This implies \eqref{eq284}.
\end{proof} 
\begin{lem}\label{lem2.2}
We have convergence in distribution
$$
\sqrt{\Delta}\nu^\Delta_\infty\Rightarrow \eta,\  \Delta\to0+,
$$
where $\eta $ is  a random variable with its distribution function given in \eqref{eq201}.
\end{lem}
\begin{proof}
 By the mean value theorem we have
$$
\forall y\in(0,1)\ \  \ \ln(1-y)=-y-\theta \frac{y^2}2,
$$
where $\theta\in (0,1).$
Let $x\geq 0$ be fixed. Then for some (another) $\theta\in (0,1)$:
$$
\ln\Pb(\nu_\infty^\Delta\geq \frac{x}{\sqrt{\Delta}})=\sum_{1\leq i\leq \frac{x}{\sqrt{\Delta}}} \ln(1-2i\Delta)=
$$
\be\label{eq317}
-\sum_{1\leq i\leq \frac{x}{\sqrt{\Delta}}} 2i\Delta -\theta \sum_{1\leq i\leq \frac{x}{\sqrt{\Delta}}} i^2\Delta^2.
\ee
Consider the first item in  \eqref{eq317}
$$
\sum_{1\leq i\leq \frac{x}{\sqrt{\Delta}}} 2i\Delta=\left[\frac{x}{\sqrt{\Delta}}\right] \left(\left[\frac{x}{\sqrt{\Delta}}\right]+1\right)\Delta\to x^2,\ \Delta\to0+.
$$
Consider the second item
$$
0\leq\sum_{1\leq i\leq \frac{x}{\sqrt{\Delta}}} i^2\Delta^2\leq \left( \frac{x}{\sqrt{\Delta}}\right)^3\Delta^2\to 0,\ \Delta\to0+.
$$
Lemma \ref{lem2.2} is proved.
\end{proof} 
Let $T_0=0, \ T_{k+1}=\inf\{j>T_k\ : \ X_j=0\},\ k\geq 0,$ be the moment of $k$th return to 0 
(we set by the definition that infimum over the empty set is equal to infinity), $T_\infty:= \sup\{k\geq 1 \ : \ X_k=0\}=\sup\{T_k\ :\ T_k\neq\infty\}$.

Denote  by $\tau_k=T_{k+1}-T_k$  the time between successive returns ($\infty-\infty:=\infty$). 
\begin{lem}\label{lem2.3}
$$
\E T_\infty=2\sum_{k=1}^{\left[ \frac{1}{2\Delta}\right]} \left(\frac1{k\Delta}- k\Delta\right) \Pb({T_{k-1}<\infty}).
$$
\end{lem}
\begin{proof}
Let $\{S_k\}$ be a random walk with  unit jumps, $p_{i,i+1}=p,\ p_{i,i-1}=q=1-p,$ $S_0=0,$
$\tau_S=\inf\{k\geq 1\ :\ S_k=0\}$ be the moment of the first return to 0.

It follows from the  definition of the RWM that the  conditional distribution of $\tau_k$
given $\{T_{k-1}<\infty\}$ coincides with the distribution of $\tau_S$ if   $p=p_k=(\frac12+k\Delta)\wedge 1.$

Recall that the moment generating function of $\tau_S$ equals, see \cite{Feller1},
$$
\E s^{\tau_S}\1_{\tau_S<\infty}=1-\sqrt{1-4pqs^2}.
$$
Therefore
\be\label{eq349}
\E{\tau_S}\1_{\tau_S<\infty}=(1-\sqrt{1-4pqs^2})'|_{s=1}=\frac{8pqs^2}{2\sqrt{1-4pqs^2}}|_{s=1}=
\frac{4pq}{\sqrt{1-4pq}}=\frac{4pq}{|p-q|}.
\ee
We have
$$
\E T_\infty=\sum_{k=1}^{\left[ \frac{1}{2\Delta}\right]} \tau_k\1_{\tau_k<\infty}\1_{T_{k-1}<\infty}.
$$
The proof of  Lemma \ref{lem2.3}  follows from \eqref{eq349} and Lemma \ref{lem2.1}.

\end{proof}

\section{The proof of \eqref{eq197}}
Let $\{X^{(n)}_k\}$ be an RWM, $X^{(n)}_k=0,$ $\Delta_n=\frac{c}{n^\alpha}, \ \alpha\in(0,1),\ T^{(n)}_\infty=\sup\{k\geq 1 \ : \ X^{(n)}_k=0\}$.

It follows from Lemma \ref{lem2.3} that 
\be\label{eq365}
\E\frac{T^{(n)}_\infty}{n}\leq 2n^{-1} \sum_{k=1}^{\left[ \frac{1}{2\Delta_n}\right]}  \frac1{k\Delta_n}=
\ee
$$
 2 n^{-1} \sum_{k=1}^{\left[ \frac{n^\alpha}{2c}\right]}  \frac{n^\alpha}{ck}=(1+o(1))2c^{-1} n^{\alpha-1}\ln\left(\left[ \frac{n^\alpha}{2c}\right]\right)=
\to 0,\ n\to\infty.
$$
It follows from the definition of $\{X^{(n)}_k\}$ that
$$
\E\left( X^{(n)}_{k+1} \, |\, \cF_k\right)= X^{(n)}_k+ (p^{(n)}_{\nu^{(n)}_k}-q^{(n)}_{\nu^{(n)}_k})= X^{(n)}_k+ (2c\nu^{(n)}_k\Delta_n)\wedge 1,
$$
where ${\nu^{(n)}_k} =|\{0\leq i\leq k\ : X^{(n)}_i=0\}|,\ p^{(n)}_i=1-q^{(n)}_i=(\frac12+i\Delta_n)\wedge 1.$

We have
\be\label{eq383}
 X^{(n)}_{k}=\sum_{i=0}^{k-1}(X^{(n)}_{i+1}-X^{(n)}_{i}) =
 \sum_{i=0}^{k-1}\left(X^{(n)}_{i+1}-\E\left( X^{(n)}_{i+1} \, |\, \cF_i\right)\right) +\sum_{i=0}^{k-1}(2c\nu^{(n)}_in^{-\alpha})\wedge 1.
\ee
Let us estimate the second summand on the right hand side of \eqref{eq383} for $k=[nt]$
$$
    ({[nt]-T^{(n)}_\infty})\left(\frac{2c\nu^{(n)}_\infty}{ n^{\alpha}}\wedge 1\right)=\sum_{i=T^{(n)}_\infty}^{[nt]-1}\left(\frac{2c\nu^{(n)}_\infty}{ n^{\alpha}}\wedge 1\right) 
\leq \sum_{i=0}^{[nt]-1}\left(\frac{2c\nu^{(n)}_i}{ n^{\alpha}}\wedge 1\right)\leq 
$$
$$
\sum_{i=0}^{[nt]-1}\left(\frac{2c\nu^{(n)}_\infty}{ n^{\alpha}}\wedge 1\right)=[nt]\left(\frac{2c\nu^{(n)}_\infty}{ n^{\alpha}}\wedge 1\right).
$$
It follows from the last inequality, Lemma \ref{lem2.2}, and  \eqref{eq365} that 
$$
\frac{ \sum_{i=0}^{[nt]-1}(2c\nu^{(n)}_in^{-\alpha})\wedge 1}{n^{1-\frac\alpha2}}\Rightarrow 2\sqrt{c}\eta t, \ n\to \infty
$$
in $D([0,\infty))$ with the topology of the uniform convergence on compact sets.

 The sequence $M^{(n)}_k:= \sum_{i=0}^{k-1}\left(X^{(n)}_{i+1}-\E\left( X^{(n)}_{i+1} \, |\, \cF_i\right)\right),\ k\geq 0$ is a martingale difference.
 It follows from the Kolmogorov inequality that
 $$
 \forall \ve>0\ \ \ \Pb\left( \max_{k=\ov{1,[nt]}}\frac{|M^{(n)}_k|}{n^{1-\frac\alpha2}}\geq \ve\right)\leq
   {n^{\alpha-2}\ve^{-2}}{\E(M^{(n)}_{[nt]})^2}= 
$$
$$
{n^{\alpha-2}\ve^{-2}}
\sum_{i=0}^{[nt]}\left(X^{(n)}_{i+1}-\E\left( X^{(n)}_{i+1} \, |\, \cF_i\right)\right)^2=
$$
$$
{n^{\alpha-2}\ve^{-2}}
\sum_{i=0}^{[nt]}\E\left( X^{(n)}_{i+1} - X^{(n)}_i+   (2c\nu^{(n)}_i\Delta_n)\wedge 1\right)^2\leq 4nt\, {n^{\alpha-2}\ve^{-2}}=
 4n^{\alpha-1}t {\ve^{-2}}\to 0,\ n\to\infty.
 $$
This yields \eqref{eq197}.

\section{The proof of \eqref{eq192} and \eqref{eq201}} 

We need the following result on the absolute continuity of the limit.
\begin{lem}\label{lemGS}
Let $\{X_n, n\geq 1\}$ and  $\{Y_n, n\geq 1\}$ be sequences of random elements given on the same probability space and taking values
in a complete separable metric space $E$.

Assume that 

1)   $Y_n\overset{{\rm P}}{\to} Y_0, n\to \infty$;

2) for each $n\geq 1$ we have the absolute continuity of the distributions
$$
\Pb_{X_n}\ll \Pb_{Y_n};
$$
3)
the sequence $\{\rho_n(Y_n), n\geq 1\}$ is uniformly integrable,
  where $\rho_n=\frac{d\Pb_{X_n}}{d\Pb_{Y_n}}$ is the Radon-Nikodym density;

4) the sequence $\{\rho_n(Y_n), n\geq 1\}$ converges in probability to a random variable $p.$

Then the sequence of distributions $\{P_{X_n}\}$  converges weakly as $n\to\infty$ to the probability measure
 $\E(p\, |\, Y_0=y) P_{Y_0} (dy)$.
\end{lem}
 Similar result was proved by Gikhman and Skorokhod, see \cite{GS_abs_nepr}. Since their formulation differs slightly  from our,  for the save of clarity we give a proof.
 
\begin{remk}\label{rem_uni} Since   $\{\rho_n(Y_n), n\geq 1\}$  are non-negative random variables and $\E\rho_n(Y_n)=1, n\geq 1,$  the uniform integrability of 
$\{\rho_n(Y_n), n\geq 1\}$  is equivalent to $\E p=1,$ where $p=\lin \rho_n(Y_n).$ So $\E(p | Y_0=y) P_{Y_0} (dy)$ is indeed a probability measure.
\end{remk}
\begin{proof} It follows from   the condition 3 of Lemma \ref{lemGS} that  for any bounded and continuous $f:E\to \mbR$ we have  
$$
\lin \int_E f dP_{X_n}=\lin \E f(X_n)= \lin \E f(Y_n) \rho_n(Y_n)=
   \E f(Y_0) p = 
$$
$$
\E\left( f(Y_0) \, \E(p\, | \, Y_0)\right)=
\int_E f(y) \E(p \,|\, Y_0=y) P_{Y_0} (dy).
$$
 Lemma \ref{lemGS} is proved.
 \end{proof}

Let $n$ be fixed, $\mu$ be the distribution  of $\{X^\Delta_0, X^\Delta_1,\dots, X^\Delta_n\}$ in $\mbR^{n+1}, $  where $\{X^\Delta_k\}$ is an
 RWM, $X^\Delta_0=0.$

Denote by 
$\nu$   the distribution of a symmetric RW  $\{S_0, S_1,\dots, S_n\}$ with  unit jumps, 
 $S_0=0,\ S_n=\xi_1+\dots+\xi_n,$ where $ \{\xi_k\}$ are i.i.d.,
$\Pb(\xi_k=\pm1)=1/2.$
Then $\mu\ll\nu$ and
$$
\frac{d \mu}{d\nu}(i_0,i_1,\dots,i_n)= \frac{\underset{k=0}{\overset{n-1}\Pi}  \left( p_{\nu_k}\1_{i_{k+1}=i_k+1} + q_{\nu_k}\1_{i_{k+1}=i_k-1}\right)}
{2^{-n}},
$$
where ${\nu_k} =|\{0\leq j\leq k\ : i_j=0\}|,\ p_i=1-q_i=(\frac12+i\Delta)\wedge 1.$

So
\be\label{eq468}
\frac{d \mu}{d\nu}(S_0,S_1,\dots,S_n)= \frac{\underset{k=0}{\overset{n-1}\Pi}  \left( p_{\nu_k}\1_{\xi_{k+1}=1} + q_{\nu_k}\1_{\xi_{k+1}=-1}\right)}
{2^{-n}}=
\ee
$$
{\underset{k=0}{\overset{n-1}\Pi}  \left(1+((2{\nu_k\Delta})\wedge 1) \xi_k \right)},
$$
where 
${\nu_k} =|\{0\leq j\leq k\ : S_j=0\}|$.

Let $M>0$ be a fixed number. Denote by  $\{X^{M, \Delta}_k\}$  a RW with modification at 0, where modifications stop changing after $[M/\sqrt{\Delta}]$-th 
hitting  0:
$$
\Pb(X^{M, \Delta}_{k+1}=X^{M, \Delta}_k+1 | X^{M, \Delta}_0,X^{M, \Delta}_1,\dots,X^{M, \Delta}_k)=
\frac12 + (\nu^{M,\Delta}_k\wedge  [M/\sqrt{\Delta}])\Delta,
$$
$$
\Pb(X^{M, \Delta}_{k+1}=X^{M, \Delta}_k-1 | X^{M, \Delta}_0,X^{M, \Delta}_1,\dots,X^{M, \Delta}_k)=
\frac12 - (\nu^{M,\Delta}_k\wedge  [M/\sqrt{\Delta}])\Delta,
$$
where $\nu^{M,\Delta}_k=|\{0\leq j \leq k\ : \ X^{M, \Delta}_j=0\}|.$

We will assume that $X^{M, \Delta}_0=0$.

Observe that restriction of the distributions 
$\1_{\nu^{M, \Delta_n}_n\leq[M/\sqrt{\Delta}]-1} \Pb_{\{{X^{M, \Delta_n}_k}, 0\leq k\leq n\}}$ 
and $\1_{\nu^{ \Delta_n}_n\leq [M/\sqrt{\Delta}]-1} \Pb_{\{{X^{M, \Delta_n}_k}, 0\leq k\leq n\}}$
are equal. 

Similarly,   let $X_\infty $ be a solution of \eqref{eq208}, $\tau_M=\inf\{t\geq 0\ : l_{X_\infty}^0(t)\geq M\}$,
and $X_{\infty,M} $ be a solution of 
\be\label{eq492}
X_{\infty,M}(t)=\sqrt{c}\int_0^t (l^0_{X_{\infty,M}} (s)\wedge M) ds+W(t), t\geq 0.
\ee
Set $\tilde \tau_M=\inf\{t\geq 0\ : l_{X_{\infty,M}}^0(t)\geq M\}$.
Then $\1_{\tilde \tau_M \geq T}\Pb_{X_{\infty,M}}= \1_{ \tau_M \geq T}\Pb_{X_{\infty}}$.

In view of Lemma \ref{lem2.2}, to prove the Theorem it is sufficient to verify the weak convergence
$\frac{X^{M, \Delta_n}_{nt}}{\sqrt{n}}\Rightarrow W(t) $ if $\alpha>1$ and
$\frac{X^{M, \Delta_n}_{nt}}{\sqrt{n}}\Rightarrow X_{\infty, M}(t) $ if $\alpha=1.$

Let us apply Lemma \ref{lemGS}.
Set $E=C([0,1]),$ $X_n=\frac{X^{M,\Delta_n}_{nt}}{\sqrt{n}}, Y_n=\frac{S_{nt}}{\sqrt{n}}, t\in[0,1],$ 
where $S_k=\sum_{i=1}^{k}\xi_k,$ $\{\xi_k\}$ 
are i.i.d., $\Pb(\xi_k=\pm1)=1/2, $
and $S_t=S_{[t]}+(t-[t])(S_{[t+1]}-S_{[t]})$.
\begin{remk}
The space $E=C([0,T])$, and hence $E=C([0,\infty)),$ can be considered similarly.
\end{remk}
Similarly to \eqref{eq468} we get the formula for the Radon-Nikodym density
$$
\frac{d P_{n^{-1/2}{X^{M,\Delta_n}_{n\cdot}} }}{dP_{n^{-1/2}S_{n\cdot}}}(\frac{S_{n\cdot}}{\sqrt{n}})= 
{\underset{k=0}{\overset{n-1}\Pi}  \left(1+2({\nu_k\wedge [M/\sqrt{\Delta_n}])\Delta_n} \xi_{k+1} \right)},
$$
where $\nu_k=|\{0\leq j\leq n \ :\ S_j=0\}|$.

It is possible, see \cite{Borodin}, to select copies $\{S^n_k\}$ of $\{S_k\}$ and a Wiener process $W$ such that 
\be\label{eq509}
 \lin \sup_{t\in[0,1]}|\frac{S^n_{[nt]}}{\sqrt{n}}- W(t)|=0,
\ \lin\sup_{t\in[0,1]}|\frac{\nu^n_{[nt]}}{\sqrt{n}}- l^0_W(t)|=0\ \ \mbox{a.s.},
\ee
where 
$\nu^n_{k}=|\{0\leq i\leq k\ :\ S^n_{i}=0\}|.$

Set $\xi^n_k:=S^n_{k}-S^n_{k-1}.$
Let us prove the following convergence in probability
$$
\lin \ln\left(\underset{k=0}{\overset{n-1}\Pi}  \left(1+2({\nu^n_k\wedge [M/\sqrt{cn^{-\alpha}}])cn^{-\alpha}} \xi^n_{k+1} \right)\right)= 
$$
$$
\begin{cases}
0, & \alpha>1,\\
  2\left(\sqrt{c}l_W^0(t) \wedge M\right)dW(t)-\int_0^1   2\left(\sqrt{c}l_W^0(t) \wedge M\right)^2dt
,& \alpha =1.
\end{cases}
$$
Consider only the case $\alpha =1$, the case $\alpha>1$ is similar and simpler.

We have
$$
{\sum_{k=0}^{n-1} \ln\left(1+2({\nu^n_k\wedge [M/\sqrt{cn^{-1}}])cn^{-1}}   \xi^n_{k+1} \right)}=
$$
$$
\sum_{k=0}^{n-1}  2\left(\frac{ \sqrt{c}\nu^n_k}{\sqrt{n}} \wedge M\right)\frac{\xi^n_{k+1}}{\sqrt{n}}   -
\sum_{k=0}^{n-1}  2\left(\frac{ \sqrt{c}\nu^n_k}{\sqrt{n}} \wedge M\right)^2\frac{1}{n}   +
\frac{\theta}{3}\sum_{k=0}^{n-1}  \left(2({\nu^n_k\wedge [M/\sqrt{cn^{-1}}])cn^{-1}}   \xi^n_{k+1} \right)^3+
o(1),
$$
where $\theta\in(0,1),$ $o(1)\to 0$ as $n\to\infty$ in probability.

The third summand converges to 0 for all $\omega.$ Indeed
$$
|\sum_{k=0}^{n-1}  \left(2({\nu^n_k\wedge [M/\sqrt{cn^{-1}}])cn^{-1}}   \xi^n_{k+1} \right)^3|\leq
\sum_{k=0}^{n-1}  \left(2({  [M/\sqrt{cn^{-1}}])cn^{-1}}   \right)^3=
$$
$$
n \left(2({  [M/\sqrt{cn^{-1}}])cn^{-1}}   \right)^3|\to0,\ n\to\infty.
$$

It follows from \eqref{eq509} that the limit of the second term is 
$\int_0^1   2\left(\sqrt{c}l_W^0(t) \wedge M\right)^2dt.$

Consider the first item.  
Let $\ve>0$ be fixed. Select $\delta>0$ and $  N\geq 1$ such that
$$
\forall n\geq N\ \ \ \Pb\left(\sup_{t\in[0,1]}|\frac{S^n_{[nt]}}{\sqrt{n}}- W(t)|+
 |\frac{\nu^n_{[nt]}}{\sqrt{n}}- l^0_W(t)|\geq\ve\right)\leq \ve;
$$
$$
\Pb\left(\sup_{s,t\in[0,1], |s-t|\leq \delta} 
 |\frac{\nu^n_{[nt]}}{\sqrt{n}}- |\frac{\nu^n_{[ns]}}{\sqrt{n}}|\geq\ve\right)\leq \ve;\  \Pb\left(\sup_{s,t\in[0,1], |s-t|\leq \delta} 
 |l^0_W(t)- l^0_W(s)|\geq\ve\right)\leq \ve.
$$
Set $m=[\frac1{\delta}]+1.$ For simplicity assume that $n/m$ is integer.  Then
$$
I_n:=\left|\sum_{k=0}^{n-1}  \left(\frac{ \sqrt{c}\nu^n_k}{\sqrt{n}} \wedge M\right)\frac{\xi^n_{k+1}}{\sqrt{n}}   -\int_0^1   \left(\sqrt{c}l_W^0(t) \wedge M\right)dW(t)\right|\leq
$$
$$
\left|\sum_{j=0}^{m-1}\sum_{k=jn/m}^{(j+1)n/m-1}  \left(\frac{ \sqrt{c}\nu^n_k}{\sqrt{n}}  \wedge M - \frac{\sqrt{c}\nu^n_{jn/m}}{\sqrt{n}}  \wedge M\right)\frac{\xi^n_{k+1}}{\sqrt{n}}\right|+
$$
$$
\left|\sum_{j=0}^{m-1}  \left(\frac{\sqrt{c}\nu^n_{jn/m}}{\sqrt{n}}  \wedge M\right)\left( \Big(\sum_{k=jn/m}^{(j+1)n/m-1}\frac{\xi^n_{k+1}}{\sqrt{n}}\Big)
- \Big(W(\frac{j+1}{m})-W(\frac{j}{m})\Big)\right)\right|+ 
$$
$$ 
 \left|\sum_{j=0}^{m-1} \Big(\frac{\sqrt{c}\nu^n_{jn/m}}{\sqrt{n}}  \wedge M- (\sqrt{c}l_W^0(\frac{j}{m}))\wedge M\Big)\Big(W(\frac{j+1}{m})-W(\frac{j}{m})\Big)\right|+
$$
$$
  \left|\sum_{j=0}^{m-1} \int_{\frac{j}{m}}^{\frac{j+1}{m}} \left((\sqrt{c}l_W^0(\frac{j}{m}))\wedge M -    (\sqrt{c}l_W^0(t)) \wedge M\right)dW(t)\right|
  =I^{n,m}_1+I^{n,m}_2+I^{n,m}_3+ I^m_4.
$$
It follows from \eqref{eq509} and Lebesgue dominated convergence theorem
that 
$$
\lin \E(I_1^{n,m})^2= \lin \frac1n\, \E \sum_{j=0}^{m-1}\sum_{k=jn/m}^{(j+1)n/m-1}  \left(\frac{ \sqrt{c}\nu^n_k}{\sqrt{n}}  \wedge M - \frac{\sqrt{c}\nu^n_{jn/m}}{\sqrt{n}}  \wedge M\right)^2=
$$
$$
\sum_{j=0}^{m-1}\E\int_{\frac{j}{m}}^{\frac{j+1}{m}} \left((\sqrt{c}l_W^0(\frac{j}{m}))\wedge M -    (\sqrt{c}l_W^0(t)) \wedge M\right)^2dt
=\E(I^n_4)^2
$$
It follows from \eqref{eq509} that  $\lin  I^{n,m}_2 =\lin   I^{n,m}_3 =0$ a.s.  for each fixed $m$. Since the second moments of
$I^{n,m}_2, I^{n,m}_3$ are uniformly bounded we have convergence
$$\forall m\geq 1\ \ \lin \E|I^{n,m}_2|=\lin \E| I^{n,m}_3|=0.$$ 
So for any $m\geq 1$
$$
\limsup_{n\to\infty}\E \left|\sum_{k=0}^{n-1}  \left(\frac{ \sqrt{c}\nu^n_k}{\sqrt{n}} \wedge M\right)\frac{\xi^n_{k+1}}{\sqrt{n}}   -\int_0^1   \left(\sqrt{c}l_W^0(t) \wedge M\right)dW(t)\right|\leq
$$
$$
 2\left( \sum_{j=0}^{m-1}\E\int_{\frac{j}{m}}^{\frac{j+1}{m}} \left((\sqrt{c}l_W^0(\frac{j}{m}))\wedge M -  
  (\sqrt{c}l_W^0(t)) \wedge M\right)^2dt \right)^{1/2}.
$$
Letting $m\to\infty$ we  get 
$$
\lin \E|I_n|=0.
$$


To apply Lemma  \ref{lemGS} it remains to prove the uniform integrability of the sequence of the Radon-Nikodym densities. However, see Remark \ref{rem_uni}, 
it is sufficient to prove 
that
$$
\E \exp\left\{  
\int_0^1   2\left(\sqrt{c}l_W^0(t) \wedge M\right)dW(t)-\int_0^1   2\left(\sqrt{c}l_W^0(t) \wedge M\right)^2dt
\right\}=1.
$$
The last equality follows from the Novikov theorem because  the integrands are bounded, see \cite[Theorem 6.1, Chapter VI]{LiptserShiryaev}.
Hence, the sequence of processes $\{\frac{X^{M,\Delta_n}_{nt}}{\sqrt{n}}\}$ converges in distribution to a process,
 whose distribution has a density
$$
\exp\left\{  
\int_0^1   2\left(\sqrt{c}l_W^0(t) \wedge M\right)dW(t)-\int_0^1   2\left(\sqrt{c}l_W^0(t) \wedge M\right)^2dt
\right\} 
$$
with respect to the Wiener measure.

By  the Girsanov theorem, this process
  is a weak solution to the SDE \eqref{eq492}.

The Theorem is proved.

\end{document}